\documentclass[12pt]{amsart}

\newtheorem{theorem}{Theorem}[section]

\newtheorem{proposition}[theorem]{Proposition}
\newtheorem{corollary}[theorem]{Corollary}

\theoremstyle{definition}
\newtheorem{definition}[theorem]{Definition}
\newtheorem{remark}[theorem]{Remark}

\numberwithin{equation}{section}

\begin{document}

\baselineskip=15.5pt

\title[Brauer and Picard groups of moduli of parabolic bundles]{Brauer
and Picard groups of moduli spaces of parabolic vector bundles on a real curve}

\author[Bhosle U. N.]{Usha N. Bhosle}

\address{Statistics and Mathematics Unit, Indian Statistical Institute,
Bangalore 560059, India}

\email{usnabh07@gmail.com}

\author[Biswas I.]{Indranil Biswas}

\address{School of Mathematics, Tata Institute of Fundamental
Research, Homi Bhabha Road, Mumbai 400005, India}

\email{indranil@math.tifr.res.in}



\date{}

\begin{abstract}
We determine the Brauer group and Picard group of the moduli space $U^{' s}_{L,par}$ of stable parabolic vector bundles
of rank $r$ with determinant $L$ on a real curve $Y$ of arithmetic genus $g \,\ge\, 2$ with at most nodes as singularities. 
\end{abstract}

\maketitle

\tableofcontents

\section{Introduction}

For any variety there is the associated Brauer group and the Picard group. Computation of these groups for various moduli spaces have 
been carried out in numerous works. We recall that the Brauer group $Br (Z)$ of a scheme $Z$ is defined by the equivalence classes of 
the Azumaya algebras over $Z$. Equivalently, $Br (Z)$ is defined by the Morita equivalence classes of \'etale locally trivial projective 
bundles on $Z$.

Here we consider an irreducible projective curve $Y$, with at most nodes as singularities, defined over the field of real numbers. Fix a 
finite subset of real nonsingular points on $Y$. Also, fix a real point $L$ of the Picard variety of $Y$. Let $U^{' s}_{L, par}$ denote 
the moduli space of stable parabolic bundles $E_*$ on $Y$ with parabolic structure, over the chosen points, of fixed numerical type 
satisfying the condition that the vector bundle underlying $E_*$ has determinant $L$. Our main aim is to compute the Brauer group of 
$U^{' s}_{L,par}$. This is carried out in Theorem \ref{brauergroup}.

We also compute the Picard group of $U^{' s}_{L,par}$; this is carried out in Theorem \ref{picUYspar}.

In Section \ref{codimpar}, we estimate the codimension of the complement of stable loci in the moduli spaces of parabolic
bundles and in the moduli stacks of parabolic bundles. 
In Section \ref{nonlocallyfree}, we estimate the codimension of the locus of non-locally free sheaves in the moduli space of parabolic 
sheaves. These estimates find use in the above mentioned computations.

\section{Moduli of parabolic bundles over a real nodal curve} \label{realcurve}

\subsection{Notation}

Let $Y$ be a geometrically irreducible projective algebraic curve of arithmetic genus $g$,
defined over ${\mathbb R}$, with at most nodes as singularities. We assume that
$g\, \geq\, 2$. Let $p: X\, 
\longrightarrow\, Y$ be the normalization of $Y$. Let $X_{\mathbb C} \,=\, X 
\times_{\mathbb R} \mathbb{C}$ (respectively, $Y_{\mathbb C}\,=\, Y\times_{\mathbb R}
{\mathbb C}$) be the irreducible projective 
complex algebraic curve obtained from $X$ (respectively, $Y$) by base change to ${\mathbb 
C}$.

The complex conjugation $\sigma\,:\, {\mathbb C}\,\longrightarrow\, {\mathbb C}$,
$c\, \longmapsto\, \overline{c}$, induces
involutive antiholomorphic automorphisms 
$$\sigma_X\,:\, X_{\mathbb C} \,\longrightarrow\, X_{\mathbb C},\ \ \, \sigma_Y\,:\,
Y_{\mathbb C}\,\longrightarrow\, Y_{\mathbb C}$$ 
such that $p\circ\sigma_X\,=\, \sigma_Y\circ p$. 
The real points of $Y$ (respectively, $X$) are precisely the fixed points of the
$\sigma_Y$ (respectively, $\sigma_X)$. We assume that $Y$ (and hence $X$) has nonsingular
real points.

The points of ${\rm Pic}(Y_{\mathbb C})$ correspond to line bundles $\xi$ on $Y_{\mathbb C}$.
For a line bundle $\xi$, any homomorphism $\eta\,:\, \xi \,
\stackrel{\cong}{\longrightarrow}\, \sigma_Y^* \overline{\xi}$ produces a homomorphism
$$\overline{\sigma_Y^*\eta}\, :\, \overline{\sigma_Y^*\xi}\,=\, \sigma_Y^* \overline{\xi}
\, \longrightarrow\, \overline{\sigma^*_Y\sigma_Y^* \overline{\xi}}\,=\, \xi\, .
$$
A line bundle $\xi$ is called \textit{real} if there is an isomorphism $\eta\,:\, \xi \,
\stackrel{\cong}{\longrightarrow}\, \sigma_Y^* \overline{\xi}$
such that $\overline{\sigma_Y^*\eta} \circ \eta \,=\, {\rm Id}_{\xi}$.
The real points of the Picard variety ${\rm Pic}(Y)$ correspond to
the line bundles $\xi$ on $Y_{\mathbb C}$ such that $\sigma_Y^*\overline{\xi}$ is isomorphic
to $\xi$. Every such $\xi$ has a real structure (recall that $Y$ has real points).

\subsection{Parabolic sheaves}\label{parasheaves}

Let $I$ be a finite subset of the locus of smooth real points of $Y$. 

\begin{definition} \label{parastructure}
A \textit{quasi-parabolic structure} at a point $x \,\in\, I$ on a torsionfree sheaf $E$
of rank $r$ on $Y$ is a flag of real vector subspaces on the fiber $E_x$ 
\begin{equation} \label{flagtype}
E_x\,= \,F_1(E_x)\,\supset\,\cdots\,\supset\, F_{l_x}(E_x)\,\supset\, F_{{l_x}+1}(E_x)\,=\, 0\, .
\end{equation}
If $c^x_i \,=\, \dim F_i(E_x)$, then
\begin{equation}\label{nx}
\overline{n}(x)\,=\, (c^x_1,\, \cdots,\, c^x_{l_x})
\end{equation}
is called the type of the flag at $x$. Note that $c^x_1\,=\, r$.

A \textit{parabolic structure} on $E$ over $I$ is a quasiparabolic structure as above at
every $x \,\in\, I$ together with a rational number $\alpha_i(x)$, called
a parabolic weight, for each subspace $F_i(E_x)$ such that
$$0 \, < \, \alpha_1(x)\, < \,\cdots \,< \,\alpha_{l_x}(x) \,< \,1.$$

Let $k_i(x) \,=\, \dim F_{i}(E_x) - \dim F_{i+1}(E_x), i\,=\, 1,\, \cdots,\, l_x$, and
\begin{equation}\label{rk}
r_i(x) \,=\, \dim E_x-c^x_i\,=\,\dim E_x - \dim F_i(E_x),\ i \,\ge\, 2;
\end{equation}
the above integer $k_i(x)$ is called the multiplicity of the parabolic weight $\alpha_i(x)$.
 Define 
$$\overline{\alpha}(x)\,:=\, (\alpha_1(x),\, \alpha_2(x),\, \cdots ,\, \alpha_{l_x}(x)),\ \
\overline{k}(x) \,:=\, (k_1(x),\, k_2(x),\, \cdots ,\, k_{l_x}(x))\, .$$

A parabolic torsionfree sheaf $E_*$ on Y with parabolic structure over $I$ 
is a torsionfree sheaf $E$ on $Y$ together with a parabolic structure on $I$. 
A parabolic torsionfree sheaf $E_*$ is called a parabolic vector bundle if
$E$ is locally free.
\end{definition}

For a parabolic torsionfree sheaf $E_*\, :=\, (E\, ,\{F_i(x)\}\, , 
\{\alpha_i(x)\})$ as above, the parabolic degree is defined to be
$$
\mbox{par-deg}(E_*)\, :=\, \text{degree}(E)+\sum_{x \in I}
\sum_{i=1}^{l_x} \alpha_i (x) k_i(x)\, .
$$
The parabolic slope is defined to be
$$
\text{par-}\mu(E_*)\, :=\,
\frac{\mbox{par-deg}(E_*)}{\text{rank}(E)}\,=\, \frac{\mbox{par-deg}(E_*)}{r}
\, \in\, {\mathbb Q}\, .
$$

\begin{definition}
For a parabolic torsionfree sheaf $E_*$, any nonzero subsheaf $F\, \subset\, E$
has an induced parabolic structure. We denote by $F_*$ the sheaf $F$ equipped with the parabolic structure
induced by $E_*$. 

A parabolic torsionfree sheaf $E_*$ is called \textit{stable} (respectively,
\textit{semistable}) if
$$
\text{par-}\mu(F_*)\, <\, \text{par-}\mu(E_*)
$$
(respectively, $\text{par-}\mu(F_*)\,\leq\, \text{par-}\mu(E_*)$)
for all subsheaves $F$ of $E$ with $1\, \leq\, \text{rank}(F)
\, <\, \text{rank}(E)$.
\end{definition}

For each $x \,\in\, I$, let $P_x \,\subset\,{\rm SL}(r, {\mathbb C})$ be the parabolic
subgroup that preserves a fixed 
filtration of type $\overline{n}(x)$ of subspaces of $\mathbb{C}^r$ (see \eqref{nx}). Let
\begin{equation}\label{F}
{\textbf F} \,:=\, \prod_{x\in I} \ {\rm SL}(r, \mathbb{C})/P_x \, .
\end{equation}
The variety ${\bf F}$ is smooth complete and rational \cite{BR}. 
Also,
$${\rm Pic}({\bf F})\, =\, \mathbb{Z}^{\sum_{x\in I}(l_x - 1)}\, .$$
We will describe a set of generators of ${\rm Pic}({\bf F})$. For each
$1\leq\, j\, \leq\, l_x$, let
$$f_j\,:\, {\rm SL}(r, {\mathbb C})/P_x\, \longrightarrow\, {\mathbb P}(\bigwedge\nolimits^{c^x_i} {\mathbb C}^r)$$ be the
morphism that sends any filtration
$$
{\mathbb C}^r \,=\, V_1\,\supset\, V_2 \, \cdots\,\supset\, V_{l_x}\,\supset\, V_{{l_x}+1}\,=\, 0$$
to the line $\bigwedge\nolimits^{c^x_j} V_j\, \subset\, \bigwedge\nolimits^{c^x_j} {\mathbb C}^r$. Define
$\xi^x_j\,:=\, f^*_j {\mathcal O}_{{\mathbb P}(\wedge^{c^x_i} {\mathbb C}^r)}(1)$.
For any $x\, \in\, I$, the group ${\rm Pic}({\rm SL}(r, \mathbb{C})/P_x)$ is generated by $\xi^x_j$,
$j \,=\, 2,\, \cdots,\, l_x$. So the group ${\rm Pic}({\bf F})$ is generated by $\xi^x_j$,
$x \,\in\, I,\ j \,=\, 2,\, \cdots,\, l_x$.

\subsection{Moduli stacks and moduli spaces of parabolic sheaves}

If $E_*$ is a semistable parabolic vector bundle over $Y$, then $E_{*, \mathbb{C}} \,=\, 
E_*\otimes _{\mathbb R} \mathbb{C}$ on $Y_{\mathbb C}$ is semistable. For a stable parabolic vector bundle 
$E_*$ over $Y$, the vector bundle $E_{*, \mathbb{C}}$ is polystable, but it may not be stable.

\begin{definition}
A parabolic vector bundle $E_*$ over $Y$ is called geometrically stable if the parabolic vector
bundle $E_{*, \mathbb{C}} $ over $Y_{\mathbb C}$ is stable.
\end{definition}

Fix integers $r\,\ge\, 2$ and $d$ together with a point $L\,\in\, {\rm Pic}^d(
Y)$. So $L$ is a real line bundle on $Y_{\mathbb C}$. Let 
$M_{par}(r,\,d)$ (respectively, $M_{par}(r,\, L)$) denote the moduli stack of parabolic vector 
bundles of rank $r$ and degree $d$ (respectively, with a fixed determinant $L$) on $Y$. Both 
of them are irreducible smooth stacks. Let $M^{gs}_{par} (r,\,d)\,\subset\, M_{par}(r,\,d)$ and 
$M^{gs}_{par} (r,\,L)\,\subset\, M_{par}(r,\,L)$ be the open substacks of geometrically stable 
parabolic vector bundles. Let $U^{'s}_{par}(r,\,d)$ and $U^{'s}_{par}(r,\,L)$ respectively 
denote their moduli spaces; these are quasi-projective smooth varieties. 

Let $U_{Y, par}(r,\,d)$ denote the moduli space of semistable parabolic sheaves of rank $r$ degree $d$ on $Y$ with a parabolic structure 
of fixed type over all $x\,\in\, I\,\subset\, Y$. Then $U_{Y, par}(r,\,d)$ is a projective seminormal variety
\cite[Theorem 1.1]{Su}. Let $U'_{Y, par}(r,\,d)$ 
denote its open subvariety corresponding to parabolic vector bundles, it is a normal quasi-projective variety. Let $U'_{Y, par}(r,\, L)$ 
denote its normal closed subvariety corresponding to parabolic vector bundles $E_*$ with a fixed determinant $det E = L$. We denote by 
$U_{Y, par}(r,\, L)$ its closure in $U_{Y, par}(r,\, d)$ with a reduced structure, we do not know if it is normal. We shall check that the 
singular set of $U_{Y, par}(r,\, L)$ has codimension at least $3$ (see Theorem \ref{codimnonfree}).

\section{Codimension of the complement of the stable locus} \label{codimpar}

\begin{definition} \label{localtype}
Let $E_{(y)}$ denote the stalk, at a node $y \,\in\, Y$, of a torsionfree sheaf $E$ of rank $r$ 
on $Y$. Then 
$$E_{(y)} \ \cong \ \mathcal{O}_y^{\oplus a(E)} \oplus m_y^{\oplus b(E)}\, ,$$
for some integers $a(E)$ and $b(E)$ with $0 \,\le\, a(E),\, b(E) \,\le\, r$, where $m$ denotes the maximal ideal
for $y$. We will call the integer $b(E)$ the local type of $E$ at $y$.

If $b_j(E)$ is the local type of $E$ at $y_j$ for $j\,=\, 1,\, \cdots,\, m$, then the $m$-tuple $\overline{b}(E) \,=\,
(b_1(E),\, \cdots, \,b_m(E))$ is called the local type of $E$. 
\end{definition}

We fix integers $r$ and $d$ such that $r\,\ge \,1$ and $d \,>>\, 0$. Let $Quot$ denote the quot scheme of quotients of
$\mathcal{O}_Y^n$ with Hilbert 
polynomial $P(m)\,=\, m r + d+r(1-g)$; set $P(0) \,=\, n$. Let
$$ \mathcal{O}_{Quot \times Y}^n\,\longrightarrow\, {\mathcal E} \,\longrightarrow\, 0$$
be the universal quotient sheaf on $Quot \times Y$.

Let $R \,\subset\, Quot$ be the subset corresponding to the torsionfree sheaves $E$ such that $H^1(E) \,=\, 0$ and $H^0(E)
\,\cong\, {\mathbb R}^n$. Let $R^0 \,\subset\, R$ denote the open subset corresponding to the locally free sheaves $E$. The subset $R^0$ is
irreducible and nonsingular \cite[Remark after Theorem 5.3']{N}. We have $\dim R^0 \,=\, r^2(g-1) +1 + {\rm dim \ PGL} (n)$.

Fix a finite subset of points $I$ in $Y$. Define 
$$
Q_{par} \,:=\, \times_{\stackrel{Quot}{x \in I}} \ {Flag}_{\overline{n}(x)} {\mathcal E}_x\, ,
$$
the fiber product over $Quot$ of relative flag scheme of type $\overline{n}(x)$. The fiber of $Q_{par}$ over a quotient $E \,\in\, Quot$
is the variety of flags of type $\overline{n}(x)$ in the vector space $E_x$. 
Let $R^0_{par}\,\longrightarrow\, R^0$ be the restriction of $Q_{par}$ to $R^0$, that is,
$$
R^0_{par} \,=\,\times_{\stackrel{R^0}{x \in I}} \ {Flag}_{\overline{n}(x)} ({\mathcal E}\vert_{x \times R^0}).
$$
Since $R^0$ is nonsingular, and the relative flag scheme is irreducible and nonsingular, it follows that
$R^0_{par}$ is irreducible and nonsingular. 

Let $R^0_{L, par}$ denote the subset of $R^0_{par}$ consisting of all quotients $E$ with determinant a fixed
line bundle $L$. Then $R^0_{L, par}$ is irreducible and nonsingular. Let $R^{0, ss}_{L, par}$ (respectively,
$R^{0, s}_{L, par}$) denote the subset of $R^0_{L, par}$ corresponding to the semistable (respectively, stable)
vector bundles. We can similarly define $R_{par},\, R_{L, par}$ and the subset $R^{ss}_{L, par}$ (respectively,
$R^{s}_{L, par}$) corresponding to the semistable (respectively, stable) torsionfree sheaves.

We first prove the following result which is of independent interest and is also needed later. We assume that $I$ is nonempty. For $Y$ 
smooth, this result is known. Our proof is on similar lines as those in \cite[Appendix]{NR} (for r\,=\,2) and
\cite[Proposition 5.1]{Su} (for $r\,\ge\, 2$). However, several modifications of the proof are needed for generalization to nodal curves.

\begin{theorem} \label{codimRpars}
Let $g \,\ge\, 2$.
$$(1) \ \ {\rm codim} \ (R^{0,ss}_{par} \setminus R^{0,s}_{par}, \,R^{0,ss}_{par}) \ \ge \ (r-1)(g-1) +1\, . $$ 
$$ (2) \ \ {\rm codim} \ (R^0_{par} \setminus R^{0,ss}_{par}, \, R^{0}_{par}) \ \ge \ (r-1)(g-1) +1\, . $$ 
$$ (3) \ \ {\rm codim} \ (R^{0,ss}_{L, par} \setminus R^{0,s}_{L, par},\, R^{0,ss}_{L, par}) \ \ge \ (r-1)(g-1) +1\, . 
$$ 
$$ (4) \ \ {\rm codim} \ (R^0_{L, par} \setminus R^{0,ss}_{L, par}, \, R^{0}_{L, par}) \ \ge \ (r-1)(g-1) +1\, . $$ 
\end{theorem} 

\begin{proof}
(1):\, Consider $E \,\in\, R^{0, ss}_{par} \setminus R^{0, s}_{par}$. There are torsionfree sheaves 
$E^1,\, E^2$ of ranks $r_1,\, r_2$ and degrees $d_1,\, d_2$ respectively such that 
$E^1,\, E^2,\, E$ fit in an exact sequence
\begin{equation} \label{E1EE2}
0 \,\longrightarrow\, E^1 \,\longrightarrow\, E \,\longrightarrow\, E^2\,\longrightarrow\, 0\, ,
\end{equation}
and for induced parabolic structures (induced from $E_*$) they have 
\begin{equation} \label{equality}
\text{par-}\mu(E^1_*)\, = \, \text{par-}\mu(E_*) \,= \, \text{par-}\mu(E^2_*) \, .
\end{equation}

For $x \,\in\, I$, let $m^x_1,\, \cdots,\, m^x_{l_x}$ be the multiplicities for induced parabolic
structure on $E^1$, we suppress $m^x_i$ if $m^x_i\,=\, 0$. Then the equality in \eqref{equality} can be written as
\begin{equation}\label{su5.3}
r d_1 - r_1 d\,\, =\,\, \sum_{x \in I} \sum_{i=1}^{l_x} (r_1 k_i(x) - r m^x_i) \alpha_i(x) 
\end{equation}

We first construct a countable number of quasiprojective varieties parametrizing extensions of type \eqref{E1EE2}. Let $n_1,\,
n_2$ be integers with $n_1+n_2\,=\, n$. For $i \,=\, 1,\, 2$, let $Q^i$ denote the Quot scheme of quotients 
$\mathcal{O}_Y^{n_i}\,\longrightarrow\, E^i \,\longrightarrow\, 0\, ,$
with ${\rm rank}(E^i)\,=\, r_i$ and ${\rm degree}(E^i)\,=\, d_i$. Let 
$$ \mathcal{O}_{Q^i \times Y} ^{n_i} \,\longrightarrow\, \mathcal{E}^i \,\longrightarrow\, 0$$
be the universal quotient over $Q^i \times Y$.
For $i \,=\, 1, \,2$, let $\overline{b}_i \,=\, ((b_{i, j})_j)$ be a string of integers with $j$ running over nodes. 
For $0 \,\le\, b_{i,j} \,\le \,r_i$, define (finitely many) locally closed subsets 
$$Q^i_{\overline{b}_i} \,\,:=\, \,\{ q_i \,\in\, Q^i\, \ \vert\, \ \overline{b}_i(\mathcal{E}^i_{q_i}) \, =\, \overline{b}_i \} \, ,$$
where $\overline{b}_i(\mathcal{E}^i_{q_i})$ denotes the local type of the sheaf $\mathcal{E}^i_{q_i}$.

Set $Q^{1,2} \,=\, Q^1 \times Q^2$. Over $Q^{1,2} \times Y$ we have pullbacks of $\mathcal{E}^i$,\, $i \,=\, 1,\, 2$,
which are denoted by the same symbol, for simplicity of notation. Let $t\, =\, (h,\, \overline{b}_1,\, \overline{b}_2)$, where
$h$ is a nonnegative integer. 
Define (countably many) locally closed subsets
\begin{equation}\label{dh}
Q^{1,2}_t \,:=\, \{ q\,=\,(q_1,\, q_2) \,\in\, Q^{1,2} \,\big\vert\,\,
h^1({\rm Hom}( \mathcal{E}^2_{q},\, \mathcal{E}^1_{q}))\,=\, h,\, 
q_i \,\in\, Q^i_{\overline{b}_i},\, i \,=\, 1,\, 2 \}\, .
\end{equation} 
Let $\phi\,: \,Q^{1,2} \times Y \,\longrightarrow\, Q^{1,2}$ be the natural projection. Denote by
${\bf Ext}^1_{\phi} (\mathcal{E}^2,\, \mathcal{E}^1)$ the relative ${\bf Ext}$ sheaf on $Q^{1,2}$ for this map (see \cite{La}
for generalities on the relative ${\bf Ext}$ sheaf). By \cite[Lemma 2.5(B)]{B5}, we have 
$$\dim {\rm Ext}^1_Y(\mathcal{E}^2_{q}, \,\mathcal{E}^1_{q})\, =\, h^1({\rm Hom}(\mathcal{E}^2_{q},\, \mathcal{E}^1_{q})) +
2 \sum_j b_{1, j} b_{2, j}.$$ Hence for $q \,\in\, Q^{1,2}_t$, we have $\dim {\rm Ext}^1_Y(\mathcal{E}^2_{q}, \,\mathcal{E}^1_{q})$
to be constant (and $\dim {\rm Ext}^0_Y(\mathcal{E}^2_{q}, \,\mathcal{E}^1_{q})$ is also constant) as $q$ varies.
Therefore, the relative ${\bf Ext}$ sheaf ${\bf Ext}^1_{\phi} (\mathcal{E}^2,\, \mathcal{E}^1)$ is locally free on $Q^{1,2}_t$.

If $h\, =\, 0$ (see \eqref{dh}), define $P_t \,:= \, Q^1_{\overline{b_1}} \times Q^2_{\overline{b_2}}$ and
$\mathcal{E}^t \,=\, \mathcal{E}^1 \oplus \mathcal{E}^2$ on $P_t$.

If $h \, >\,0$, define $P_t \,:=\, \mathbb{P}({\bf Ext}^1_{\phi} (\mathcal{E}^2,\, \mathcal{E}^1)^*)$, a projective bundle
on $Q^{1,2}_t$. On $P_t \times Y$, we have the universal extension 
$$0 \,\longrightarrow\, \mathcal{E}^1 \otimes \mathcal{O}_{P_t}(1) \,\longrightarrow\, \mathcal{E}^t
\,\longrightarrow\, \mathcal{E}^2 \, \longrightarrow 0\, .$$
Let $P'_t \,\subset \,P_t$ be the open subset corresponding to $p_t \,\in\, P_t$ such that $\mathcal{E}^t\vert_{p_t \times Y}$
is locally free. The quasiprojective variety $P'_t$ parametrizes extensions of type \eqref{E1EE2}.

For each $x \,\in\, I$, let $u(x)\,=\, (r_1,\, d_1,\, t,\, m^x_1,\, \cdots,\, m^x_{l_x})$, where
$t \,=\, (h, \,\overline{b}_1, \,\overline{b}_2)$. Define locally closed subschemes
$$S_{u(x)} \,\subset\, {\mathcal F}lag_{\overline{n}(x)} {\mathcal E}^t_x\, ,$$ 
which are fibrations over $P'_t$ whose fibers $S^0_{u(x)}$ consist of flags 
$$E_x\, =\, F_1(E_x) \,\supset\, \cdots \,\supset\, F_{l_x}(E_x) \,\supset\, F_{{l_x}+1}(E_x) \,=\, 0$$
such that
$$\dim \, (F_i(E_x) \cap E_{1,x}) \ = \ r_1 - \sum_{j = 1}^i m^x_j\, .$$
Let
$$S_u \,:=\, \ \times_{\stackrel{P'_t}{x \in I}} \ S_{u(x)}\, .$$
We have 
$$\dim \, S_u \,\le\, \ {\rm dim} \ P_t + \sum_{x \in I} \, \dim \, S^0_{u(x)}\, .$$
Each $S_u$ parametrizes a family of parabolic sheaves $E$ which occur as extensions of type \eqref{E1EE2} (possibly split)
with parabolic structures at $x \,\in\, I$ of type $\overline{n}(x)$ such that the induced structures on $E^1$ are of type
$\overline{m}(x)\, =\, (m^x_1,\, \cdots,\, m^x_{l_x})$ (recall that we suppress $m^x_i$ if $m^x_i \,=\,0$).

Using Riemann-Roch theorem and \cite[Lemma 2.5(B)]{B5}, we have that 
$$\dim {\rm Ext}^1_Y(\mathcal{E}^2_{q}, \,\mathcal{E}^1_{q})
\,=\, r_1 r_2(g-1) + \sum_j b_{1, j} b_{2, j} + h^0({\rm Hom}(\mathcal{E}^2_{q}, \,\mathcal{E}^1_{q}))\, .$$
As in the proof of \cite[Proposition 2.7]{B5}, one has $$\dim Q^i_{\overline{b_i}} \,\le\, r_i^2(g-1) +1 - \sum_j b_{i,j}^2 
+ \ \dim\, {\rm PGL}(n_i).$$ It follows that 
$$
 \begin{array}{ccl}
{\rm dim} \ P_t & \le & (g-1)\sum_{i=1}^2 r_i^2 + 2 - \sum_j b_{1,j}^2 - \sum_j b_{2,j}^2 + \sum_{i=1}^2 \ \dim {\rm PGL}(n_i) \\
{} & {} & + r_1 r_2(g-1) + \sum_j b_{1, j} b_{2, j} + h^0({\rm Hom}(\mathcal{E}^2_{q}, \mathcal{E}^1_{q})).
\end{array}
$$

Let
\begin{equation}\label{sus}
S^{ss}_u \,\subset\, S_u
\end{equation}
be the subset corresponding to the semistable parabolic bundles. Let $\mathcal{E}^u$ denote the pullback
of $\mathcal{E}^t$ to $S^{ss}_u$ (see \eqref{sus}). Denote by $F^u$ the frame bundle of the direct image of $\mathcal{E}^u$ to
$S^{ss}_u$; it is a principal-${\rm GL}(n)$-bundle. There is a morphism 
$$\psi_u\,:\, F^u \,\longrightarrow\, R^{0, ss}_{par} \setminus R^{0, s}_{par}\, .$$
The union of $\psi_u(F^u)$ (as $u$ varies) covers $R^{0, ss}_{par} \setminus R^{0, s}_{par}$. Let $c$ be the infimum of dimensions
of irreducible components of the fibers of $\psi_u$. Since $E \,=\, \mathcal{E}^t_q$ is globally generated by sections, any element of
${\rm Aut}(E)$ acts nontrivially on $H^0(E)$. If $h\,=\,0$ (see \eqref{dh}), then $\dim\, {\rm Aut}(E) \,\ge\, 2 + h^0({\rm Hom}(E^2,
\,E^1))$, and if $h\, >\,0$, then $\dim\,{\rm Aut}(E)\, \ge\, 1 + h^0({\rm Hom}(E^2,\, E^1))$, so that
$$c \,\ge\, h^0({\rm Hom}(E^2,\, E^1)) + n_1^2 + n_2^2\, \ {\rm if}\, \ h\,=\,0\, ,$$
$$c \,\ge\, h^0({\rm Hom}(E^2,\, E^1)) + n_1^2 + n_2^2 - 1\, \ {\rm if} \,\ h \,>\, 0\, .$$
Therefore,
$$
\begin{array}{lll}
\dim \, \psi_u(F^u) & \,=\, & \dim \, S^{ss}_u + n^2 - c \\ 
{} & \,\le\, & n^2 + (g-1)\sum_{i=1}^2 r_i^2 + r_1 r_2(g-1) + \sum_j b_{1, j} b_{2, j} -\sum_j b_{1,j}^2 - \sum_j b_{2,j}^2 \\
{} & {} & + \sum_{x\in I} \ \dim\, S^0_{u(x)} \, .$$ 
\end{array}
$$
Hence the following holds:
 $$
 \begin{array}{lll}
{\rm codim} \ \psi_u(F^u) & \ge & r^2(g-1) +1 + \ {\rm dim \ PGL} (n) + \ {\rm dim} \ {\bf F} -{\rm dim} \ \psi_u(F^u) \\
{} & = & r_1 r_2 (g-1) + \sum_j (b_{i,j} - b_{2,j})^2 + \sum_j b_{1, j} b_{2, j} \\
{} & {} & + \sum_{x \in I} \ {\rm codim} \ S^0_{u(x)}\\
{} & \ge & r_1 r_2 (g-1) + \sum_{x \in I} \ {\rm codim} \ S^0_{u(x)} \\
{} & \ge & (r -1)(g-1) + \sum_{x \in I} \ {\rm codim} \ S^0_{u(x)}\, .
\end{array}
$$
The codimension of $S^0_{u(x)}$ is given by \cite[Lemma 5.1]{Su}. Then equation \eqref{su5.3} and \cite[Lemma 5.2]{Su}
together give that ${\rm codim} \ \psi_u(F^u)\,\ge\, (r-1)(g-1) + 1$. 
Since $\psi_u(F^u)$ (as $u$ varies) cover $R^{0, ss}_{par} \setminus R^{0,s}_{par}\, ,$
it follows that
$${\rm codim} \ (R^{0, ss}_{par} \setminus R^{0,s}_{par}, \,R^{0,ss}_{par}) \ \ge \ (r-1)(g-1) +1\, . 
$$

(2):\, Part (2) can be proved similarly as Part (1) is done. We consider $E \,\in\, R^{0} \setminus R^{0, ss}$.
There are torsionfree sheaves $E^1,\, E^2$ of ranks $r_1,\, r_2$ and degrees $d_1,\, d_2$ respectively such that 
\begin{equation}\label{inequality}
\frac{d_1}{r_1}\, > \,\frac{d_2}{r_2},
\end{equation}
and $E^1, \,E^2$ fit in the extension \eqref{E1EE2}. As in part (1), we construct a countable number of quasiprojective
varieties parametrizing such extensions. We construct the projective bundle $P_t$, the frame bundle $F^u$ and the maps 
$$\psi_u\,:\, F^u \,\longrightarrow\, R^{0} \setminus R^{0,ss}\, ,$$
whose images (as $u$ varies) cover $R^{0} \setminus R^{0,ss}\, .$
 
In this case we have ${\rm deg}({\rm Hom} (\mathcal{E}^2_{q}, \,\mathcal{E}^1_{q}) ) \,=\, r_2 d_1 - r_1 d_2 
 +\sum_j b_{1, j} b_{2, j}$, and hence
$${\rm dim \ Ext}^1_Y(\mathcal{E}^2_{q},\, \mathcal{E}^1_{q}) 
\,=\, r_1 r_2(g-1) + r_1 d_2 - r_2 d_1 +\sum_j b_{1, j} b_{2, j} + h^0({\rm Hom} (\mathcal{E}^2_{q}, \mathcal{E}^1_{q}))\, .$$
Consequently, 
$$
 \begin{array}{ll}
{\rm dim} \ \psi_u({\mathcal F}_u) & \,\le\, n^2 + (g-1)\sum_{i=1}^2 r_i^2 + r_1 r_2(g-1) + r_1 d_2 - r_2 d_1 \\
{} & + \sum_j b_{1, j} b_{2, j} - \sum_j b_{1,j}^2 - \sum_j b_{2,j}^2 + \sum_{x\in I} \ \dim\, S^0_{u(x)}\, .
\end{array}
$$
Then 
$${\rm codim} \ \psi_u(F^u)\,\ge\, (r-1)(g-1) - r_1 d_2 + r_2 d_1 + \sum_{x \in I} \ {\rm codim} \ S^0_{u(x)}\,\ge\, (r-1)(g-1) +1$$
as in part (1), using \eqref{inequality}.
This completes the proof of Part (2).

(3):\, Part (3) follows from Part (1) using the determinant map.

(4):\, Similarly, Part (4) follows from Part (2) using the determinant map.
\end{proof}

As $U^{'}_{Y, par}(r, \,d)$ (respectively, $U^{'}_{Y, par}(r,\, L))$ is a geometric invariant
theoretic quotient of $R^{0, ss}_{par}$ (respectively, $R^{0, ss}_{L, par}$), 
the following corollary is an immediate consequence of Theorem \ref{codimRpars}.

\begin{corollary}\label{ucodim}
Assume that $g\, \geq\, 2$. Then the following statements hold:
$$(1) \ \ {\rm codim} \ (U^{'}_{Y, par}(r,\, d) \setminus U^{'s}_{Y, par}(r,\, d), \,U^{'}_{Y, par}(r,\,
d)) \ \ge \ (r-1)(g-1) +1\, . $$
$$(2) \ \ {\rm codim} \ (U^{'}_{Y, par}(r,\, L) \setminus U^{'s}_{Y, par}(r,\, L),\, U^{'}_{Y, par}(r,\, L)) \ \ge \ (r-1)(g-1) +1\, . $$
\end{corollary} 

\begin{corollary} \label{codimMpargs}
Assume that $g\, \geq\, 2$. Then the following statements hold:
$$ (1) \ \ {\rm codim} \ (M_{par} (r, \,d)\setminus M^{gs}_{par}(r,\, d), \,M_{par}(r,\,d)) \ \ge \ (r-1)(g-1) +1\, . $$ 
$$ (2) \ \ {\rm codim} \ (M_{par}(r,\, L) \setminus M^{gs}_{L, par},\, M_{par})(r,\, L) \ \ge \ (r-1)(g-1) +1\, . $$ 
\end{corollary}

\begin{proof}
This follows from Theorem \ref{codimRpars} as $R^0_{par}(r,\, d)$, $R^0_{L, par}(r,\, d)$, $R^{0, s}_{par}(r,\, d)$ and
$R^{0, s}_{L, par}(r,\, d)$
are smooth atlases for $M_{par}(r,\, d)$, $M_{par}(r, \,L)$, $M^{gs}_{par}(r,\, d)$ and $M^{gs}_{par}(r,\, L)$ respectively.
\end{proof}

\section{Picard group of the parabolic moduli space}

Our aim in this section is to compute the Picard group of the stable parabolic moduli space $U^{'s}_{Y, par}(r,\,L)$. We first compute 
the Picard group of the parabolic moduli stack and deduce from it that of the moduli space. Since the geometrically stable parabolic 
vector bundles $E_*$ have only scalar automorphisms, the stack $M^{gs}_{par}(r,\, L)$ is a gerbe with band $\mathbb{G}_m$ over $U^{' 
s}_{Y, par}(r,\, L)$.

Consider the universal bundle over $Y\times M_{par}(r,\, L)$.
Let $L(det)$ denote the corresponding determinant of cohomology line bundle on 
$M_{par}(r,\, L)$.
Let $L_{\mathbb C}\,=\, L\otimes \mathbb C$ be the line bundle on $Y_{\mathbb C}$ obtained by
base change of $L$ from $\mathbb R$ to $\mathbb C$.
Let $\widetilde{L}_{\mathbb C}$ be the line bundle over $M_{par}(r,\, L_{\mathbb C})$ whose 
fiber over a point corresponding to a parabolic vector bundle $E_*$ is ${\rm Hom}(L_{\mathbb C}, 
\, \det E)$ . Trivializing the fiber of $L_{\mathbb C}$ over a non-singular real point $p_0 \,\in\, 
Y_{\mathbb C}$, we can identify $\widetilde{L}_{\mathbb C}$ with the line bundle whose 
fiber over $E_{*, \mathbb{C}}$ is $\det (E_{*, \mathbb{C}})_{p_0}$.

Note that the line bundle $\widetilde{L}_{\mathbb C}$ is real.

\begin{proposition} \label{picMpar}
The Picard group of $M_{par}(r,\, L)$ is generated by
$L(det)$, $\widetilde{L}_{\mathbb C}$ and the generators $\xi^x_j$,\, $x \,\in\, I,\, j\,=\, 2,\, \cdots,\, l_x$ of ${\rm Pic}({\bf F})$.

The restrictions of these line bundles generate ${\rm Pic}(M^{gs}_{par}(r,\, L))$.
\end{proposition}

\begin{proof}
Let $\mathcal{SL}_{Y_{\mathbb C}, par}(r,\, d)$ denote the moduli stack of parabolic vector
bundles $E$ of rank $r$ and degree $d$ on $Y_{\mathbb C}$ together with an isomorphism
$\delta\,:\, \det E \,\stackrel{\cong}{\longrightarrow}\, L_{\mathbb C}$. It is a
${\mathbb G}_m$-torsor over ${M}_{par}(r,\, L_{\mathbb C})$ given by the line bundle
$\widetilde{L}_{\mathbb C}$. Let $f_p\,:\, \mathcal{SL}_{Y_{\mathbb C}, par}(r,\, d)
\,\longrightarrow\, {M}_{par}(r,\,L_{\mathbb C})$ be the canonical (forgetful) map
that forgets the isomorphism $\delta$. The homomorphism
$$f_p^*\ : \ {\rm Pic}({M}_{par}(r,\, L_{\mathbb C})) \,\longrightarrow\, 
{\rm Pic}(\mathcal{SL}_{Y_{\mathbb C}, par}(r,\, d))$$
induced by this canonical map is surjective and its kernel is generated by $\widetilde{L}_{\mathbb C}$.

By \cite[Theorem 6.1]{B6}, the Picard group of the stack $\mathcal{SL}_{Y_{\mathbb C}}(r,\,d)$ is isomorphic to $\mathbb{Z} \times \ {\rm Pic}
({\bf F})$, and it is generated by the pull back of the line bundle $L(det)_{Y_{\mathbb C}}$ on $M(r,\, L_{\mathbb C})$ and the generators of ${\rm 
Pic}({\bf F})$. Thus the Picard group of $M_{par}(r,\,L_{\mathbb C})$ is generated by the line bundles $L(det)_{Y_{\mathbb C}}, 
\widetilde{L}_{\mathbb C}$ and the generators of ${\rm Pic} ({\bf F})$. Note that the generators of ${\rm Pic}({\bf F})$ are all real line 
bundles. Since ${\rm Pic}(M_{par}(r,\, L))$ is the subgroup of real line bundles in ${\rm Pic}(M_{par}(r,\, L_{\mathbb C}))$, it follows that
${\rm Pic}(M_{par}(r,\, L))$ is generated by $L(det),\, \widetilde{L}$ and $(\xi^x_j)_{x, j}$.

Since $M_{par}(r,\, L)$ is a smooth stack (it has a smooth atlas $R^0_{par}$ \cite[subsection 3.1]{B6}), the restriction map 
$$Res: {\rm Pic}(M_{par}(r,\, L))\, \longrightarrow \ {\rm Pic}(M^{gs}_{par}(r,\, L))$$
is surjective. Hence restrictions of these line bundles generate ${\rm Pic}(M^{gs}_{par}(r,\, L))$.
\end{proof}

\begin{theorem}\label{picUYspar}
Assume that either $g \,\ge\, 3$ or $g \,=\,2,\, r \,\ge\, 3$. Let $\chi = d + r (1-g)$. Then 
the Picard group of $U^{'s}_{Y, par}(r,\, L)$ can be identified with the subgroup of ${\rm Pic}(M_{par}(r,\,L))$
consisting of elements of the form 
\begin{equation}\label{c11}
L(det)_Y ^{\otimes a} \bigotimes\widetilde{L}^{\otimes b} \bigotimes \otimes_{x, j} (\xi^x_j) ^{\otimes d^x_j}\, , 
\end{equation}
with $a,\, b,\, d^x_j \,\in\, \mathbb{Z}$ satisfying the only relation 
\begin{equation}\label{c12}
a \chi + b r + \sum_{x, j} d^x_j c^x_j \,=\, 0\, .
\end{equation}
\end{theorem}

\begin{proof}
Under the assumptions of the theorem, the open subset $M^{gs}_{par}(r, \,L)$ consisting of geometrically stable bundles in the smooth 
stack $M_{par}(r, \,L)$ is of codimension at least two (see Corollary \ref{codimMpargs}). Consequently, the surjective restriction map 
$Res$ is also injective and ${\rm Pic}(M^{gs}_{par}(r, \,L)) \,=\, {\rm Pic}(M_{par}(r,\, L))$. Hence ${\rm Pic}(M^{gs}_{par}(r, \,L))$ 
is generated freely by the line bundles $L(det)_Y,\, \widetilde{L}$ and the generators of ${\rm Pic}({\bf F})$; this follows using 
\cite[pp.~ 499--500, Theorem]{LS}. The line bundles $L(det)_Y$, $\widetilde{L}$ and $\xi^x_j$,\, $x\,\in \,I$,\, $j\,=\,2,\, \cdots,\, 
l_x$ are all real.

The line bundles on $M^{gs}_{par}(r,\, L)$ descend to $U^{'s}_{Y, par}(r,\, L)$ if and only if they have weight $0$. We recall that the 
weight of a line bundle $N$ on the irreducible stack $M^{gs}_{par}(r,\, L)$ is $\ell\, \in\, \mathbb Z$ if for any
$E_*\, \in\, M^{gs}_{par}(r,\, L)$, and any $\lambda\, \in\, {\mathbb G}_m$, the automorphism of $E_*$ given by the multiplication with
$\lambda$ acts on the fiber $N_{E_*}$ as multiplication by $\lambda^\ell$.
Now $L(det)_Y$ has weight $\chi$, while $\widetilde{L}$ has weight $r$ 
and the line bundle $\xi^x_j$ has weight $c^x_j$ for $x\, \in\, I$,\, $j\,=\, 2, \,\cdots ,\, l_x$. Therefore, the theorem follows.
\end{proof}

\begin{remark}\label{picULpar}
Note that ${\rm Pic}(U^{'s}_{Y_{\mathbb C}, par}(r, \, L_{\mathbb C}))$ consists of elements of the form described in Theorem \ref{picUYspar}.
\end{remark}

\section{Codimension of the locus of non locally free sheaves} \label{nonlocallyfree}

In this section, we estimate the codimension of the closed subset in $U_{Y, par}(r,\,d)$ consisting of 
torsionfree sheaves which are not locally free. This is done using parabolic GPBs. 

\subsection{Parabolic GPB}\label{paraGPBs}

For $j \,=\, 1,\, \cdots,\, m$, fix divisors $D_j \,=\, x_j + z_j, x_j \,\neq\, z_j\, ,$ where $(x_j,\, z_j)$ are distinct pairs of
closed points of the normalization $X$ of $Y$. Let $I$ be a set of distinct closed points of $X$ which are distinct from $x_j,\, z_j,\, j \,=\, 1,\, \cdots,\, m$.

\begin{definition}\label{GPB}
A generalized parabolic bundle (GPB in short) $\overline{E}\,=\, (E, F(E))$ of rank $r$ and degree $d$ on $X$ is a vector bundle $E$ of rank $r$ and degree $d$ on $X$ with a GPB structure over the divisors $D_j$, i.e., an $m$-tuple $F(E)\,=\, (F_1(E), \,\cdots,\, F_m(E))$, where
$F_j(E) \,\subset\, E_{x_j} \oplus E_{z_j}$ is a vector subspace of dimension $r$ for $j\,=\, 1,\, \cdots,\, m$. 
\end{definition}

\begin{definition} \label{paragpb}
A parabolic GPB $\overline{E}_*\,=\, (E_*, \,F(E))$ on $X$ of rank $r$ and degree $d$ is a parabolic vector bundle $E_*$ of rank $r$ and
degree $d$ on $X$, with parabolic structure at points of $I$, together with a GPB structure over the divisors $D_j$. \end{definition}

We may define the determinant of a GPB as a GPB of rank $1$. There is a bijective correspondence between line bundles $L$ on $Y$
and GPBs $\overline{L}$ of rank $1$ on $X$ (see \cite[Section 3]{BS} for more details). 
For a real number $0 \,\le\, \alpha \,\le\, 1$, there is the notion of $\alpha$-(semi)stability of a GPB and
of a parabolic GPB. We denote by $P^{par}(r,\, \overline{L})$ the moduli space of $\alpha$-semistable parabolic GPBs $(E_*,\, F(E))$ on $X$
of rank $r$, degree $d$ with $\alpha$ sufficiently close to $1$. It is a normal projective variety of dimension $(r^2 - 1)(g-1) + \, \dim \,
{\bf F}$. There is a surjective morphism
\begin{equation}\label{fp}
f_{par}\,:\, P^{par}(r, \,\overline{L}) \,\longrightarrow\, U_{Y, par}(r,\, L)\, ,
\end{equation}
which is an isomorphism over $U'_{Y, par}(r,\, L)$. 

To a parabolic GPB $(E_*,\, F(E))$, of rank $r$ and degree $d$ on $X$, the map $f_{par}$
in \eqref{fp} associates a parabolic torsion-free sheaf $F_*$ of rank $r$ and degree $d$ on $Y$ given by the extension
$$0 \,\longrightarrow\, F_* \,\longrightarrow\, p_* E_* \,\longrightarrow\, \oplus_j\ p_* \frac{(E_{x_j} \oplus E_{z_j})}{F_j(E)}
\,\longrightarrow\, 0\, .$$ 
Let $p_{x_j}\,:\, F_j(E)\,\longrightarrow\, E_{x_j}$ and $p_{z_j}\,:\, F_j(E)\,\longrightarrow\, E_{z_j}$ be the two
projections. Then the local type $b_j(F)$ of $F$ at the node $y_j$ is given by 
$$b_j(F) \,= \ {\rm dim} \ ker \ p_{x_j} + \ {\rm dim} \ ker \ p_{z_j}\, .$$

For $\overline{s}\,=\,(s_1,\,\cdots,\,s_m)$ and $\overline{t}\,=\,(t_1,\,\cdots,\,t_m)$, where $s_j$ and $t_j$ are integers with
$0 \,\le\, s_j \,\le\, r$ and $0 \,\le\, t_j \,\le\, r$, define $P^{par}_{\overline{s},\overline{t}}(r, \,\overline{L})$ by 
$$P^{par}_{\overline{s},\overline{t}}(r,\, \overline{L}) \ := \ \{ (E,\, F(E)) \,\in\, P^{par}(r,\,\overline{L}) \,\,\big\vert\, \
{\rm dim} \ ker \ p_{x_j}\, =\, s_j, \ {\rm dim} \ ker \ p_{z_j} \,=\, t_j\}\, .$$
For $\overline{b}\,=\,(b_1,\,\cdots,\,b_m)$, where $b_j$ are integers with $0 \,\le\, b_j \,\le\, r$, let 
$$
P^{par}_{\overline{b}}(r,\,\overline{L}) \,=\, \bigcup_{\overline{s}+\overline{t}=\overline{b}}
\ P^{par}_{\overline{s},\overline{t}}(r,\,\overline{L})\, .
$$ 
The moduli space $P(r, \,\overline{L})$ is the union of $P_{\overline{b}}(r,\,\overline{L})$. The image
$f_{par}(P^{par}_{\overline{b}}(r,\,\overline{L}))$ consists of semistable torsion-free sheaves which are of local type $b_j$ at $y_j$ for
all $j$, and $U_{Y, par}(r, \,L)$ is the union of the subsets $f_{par}(P^{par}_{\overline{b}}(r,\,\overline{L})$. In particular, for
$\overline{0}\,=\,(0,\,\cdots,\, 0)$, this $P^{par}_{\overline{0}}(r,\,\overline{L})$ maps isomorphically onto $U'_{Y, par}(r,\, L)$. 

Since the determinant $L$ is fixed and is locally free, for $(E, \,F(E)) \,\in\, P^{par}(r,\, \overline{L})$, either $p_{x_j}$ and $p_{z_j}$
are both isomorphisms or neither of them is an isomorphism (Case (ii) in the proof of \cite[Proposition 3.3]{B02}). Hence one
does not have $s_j \,=\, 0$,\, $t_j \,\neq\, 0$ or $s_j \,\neq\, 0$,\, $t_j\,=\,0$, so that $b_j \,\ge\, 2$ if nonzero.

\begin{theorem}\label{codimnonfree}
Let $Y$ be an irreducible nodal curve of arithmetic genus $g \ge 2$ with at least one node. Then 
$${\rm codim} \ (U_{Y, par}(r,\, L) \setminus U'_{Y, par}(r,\, L),\, U_{Y, par}(r,\, L)) \ \ge\, 3$$ for $r\,\ge\, 2$.
\end{theorem}

\begin{proof}
A major part of the proof is along similar lines as that of \cite[Theorem 1.3]{BS}. To avoid repetition, we only give the modifications 
needed. We first note that the parabolic moduli space $U_{X, par}(r,\, p^*L)$ has dimension $(r^2-1)(g(X)-1) + \ {\rm dim} \ {\bf F}$ for $g(X) 
\,\ge\, 1$. Hence we only have to consider two cases viz., $g(X) \,\ge\, 1$ and $g(X) \,=\, 0$.

In case $Y$ is not rational (i.e., $g(X) \,\ge\, 1$), the proof is exactly the same as that of \cite[Theorem 3.11]{BS} with
\begin{itemize}
\item $M_X(r, \,\pi^*L_0)$ replaced with $U'_{X, par}(r, \,p^*L)$,

\item $\overline{M}(r,\, L_0)$ replaced with $U_{Y, par}(r,\, L)$,

\item $M(r,\, L_0)$ replaced with $U'_{Y, par}(r,\, L)$, and

\item the subsets $P_{s,t}(r, \, \overline{L}_0)$ and $P_{\overline{b}}(r, \,\overline{L}_0)$ of the moduli space of GPBs replaced with
the subsets $P^{par}_{s,t}(r,\, \overline{L})$ and $P^{par}_{\overline{b}}(r, \,\overline{L})$ respectively of the moduli space of parabolic GPBs.
\end{itemize}
Since the set $I$ consists of nonsingular points of $Y$, the parabolic structure on the parabolic GPB $(E_*, \,F(E))$, and the parabolic
structure on $F_*$ on $Y$ given by the GPB, are determined by each other. Hence the dimension $d_f$ of the fiber over $F_*$ is given by
the same formula as in the non-parabolic case. Then calculations similar to those in the non-parabolic case give that
$${\rm codim} \ f_{par}(P^{par}_{\overline{b}}(r,\, \overline{L}))\ \ge \ \frac{3 r b_j}{4}\, .$$
Since $b_j \,\ge\, 2$, it follows that
$${\rm codim} \ (U_{Y, par}(r,\, L) \setminus U'_{Y, par}(r,\, L), U_{Y, par}(r, \,L)) \ \ge\, 3 \ {\rm for}\, \ r \,\ge\, 2,\
g(X) \,\ge\, 1\, .$$ 

Now we come to the case where $Y$ is a rational curve (i.e., $g(X) \,=\, 0$). If $(E_*,\, F(E))$ is a semistable parabolic GPB on
$\mathbb{P}^1$, then the parabolic semistability implies that there are only finitely many choices of the underlying bundle $E$. Therefore,
there are at most finitely many irreducible components of $P^{par}(r,\, \overline{L})$ each of dimension 
$$m(r^2 - 1) + {\rm dim} \ 
{\bf F} - \ {\rm P} ({\rm Aut}(E)) \,\le\, {\rm dim} \ {\bf F} + (m-1) (r^2 - 1)$$ 
(by the proof of \cite[Proposition 3.10]{BS}). Then one has 
$${\rm dim} \ P^{par}_{\overline{s},\overline{t}}(r,\, \overline{L})\ \le \ {\rm dim} \ {\bf F} +(r^2 -1)(m-1) + r^2 - s^2 -1\, ,$$
so that 
$$
{\rm dim} \ f_{par}(P^{par}_{\overline{b}}(r, \,\overline{L})) \ \le \ {\rm dim} \ {\bf F} +(r^2 -1)(m - 1) - 
\frac{3r b_j}{4}\, .$$
Since ${\rm dim} \ U_{Y, par}(r, \,L) \ = \ {\rm dim} \ {\bf F} + (r^2 -1) (g-1)$, and $g \,=\, m$, this implies that 
$${\rm codim} \ f_{par}(P^{par}_{\overline{b}}(r,\, \overline{L})) \ \ge \ \frac{3 r b_j}{4} \ \ge\ 3\, \ {\rm for} \ r \,\ge\, 2\, .$$ 
Thus, ${\rm codim} \ (U_{Y, par}(r,\, L) \setminus U'_{Y, par}(r,\, L),\, U_{Y, par}(r,\, L)) \ \ge\ 3$ for $r \,\ge\, 2$ and $g(X)\, =\, 0$. 
\end{proof} 

\begin{corollary}\label{functions}
Let $g \,\ge\, 2$ and $r \,\ge\, 2$. Then the following two hold:
\begin{enumerate}
\item $H^0(U^{' }_{Y_{\mathbb C}, par}(r, \,L_{\mathbb C}), \,\mathcal{O})\,= \,\mathbb{C}$.
\item $H^0(U^{' s}_{Y_{\mathbb C}, par}(r, \,L_{\mathbb C}), \,\mathcal{O})\,= \,\mathbb{C}$.
\end{enumerate}
\end{corollary}

\begin{proof}
(1):\, Let $h\,:\, N \,\longrightarrow\, U_{Y_{\mathbb C}, par}(r, \,L_{\mathbb C})$ be the normalization. Since
$U^{' }_{Y_{\mathbb C}, par}(r, \,L_{\mathbb C})$ is normal, it follows that $h$ is an isomorphism over $U^{' }_{Y_{\mathbb C}, par}(r,
\,L_{\mathbb C})$. As the normalization map is finite, we have 
$${\rm codim}\,(N \setminus h^{-1}(U'_{Y_{\mathbb C}, par}(r, \,L_{\mathbb C})), N)\,= \, {\rm codim}\, ( U_{Y_{\mathbb C}, par}(r, 
\,L_{\mathbb C}) \setminus U'_{Y_{\mathbb C}, par}(r, \,L_{\mathbb C}), U_{Y_{\mathbb C}, par}(r, \,L_{\mathbb C})) \,\ge \,3$$ by Theorem 
\ref{codimnonfree}. The variety $N$ being normal, this implies that the functions on $h^{-1}(U'_{Y_{\mathbb C}, par}(r, \,L_{\mathbb C}))$ extend 
uniquely to functions on $N$ and hence are constant. It now follows that the
functions on $U'_{Y_{\mathbb C}, par}(r, \,L_{\mathbb C})$ are constant.

(2):\, This can be proved similarly as Part (1) using Corollary \ref{ucodim} and Theorem \ref{codimnonfree}.
\end{proof}

\section{Brauer group of the parabolic moduli space} 

Since $U^{' s}_{Y, par}(r,\,L)$ is a smooth quasi projective variety over a field, the Brauer group ${\rm Br}(U^{' s}_{Y, par}(r, \,L))$ 
of $U^{' s}_{Y, par}(r, \,L)$ is the cohomological Brauer group $H^2_{{\rm et}}(U^{' s}_{Y, par}(r,\, L), \,\mathbb{G}_m)$; the latter is a 
torsion group.

\subsection{The Brauer class $\beta$}

\begin{definition}\label{brauerclass}
Let $\beta \,\in\, {\rm Br}(U^{' s}_{Y, par}(r, \,L))$ be the Brauer class given by the gerbe $$M^{gs}_{par}(r,\, L)\, \longrightarrow\, 
U^{' s}_{Y, par}(r, \,L)$$ with band $\mathbb{G}_m$.

Let $\mathbb{Z} \beta \,\subseteq\, {\rm Br}(U^{' s}_{Y, par}(r,\, L))$ denote the subgroup generated by $\beta$.
\end{definition}

Choosing a nonsingular closed point $p_0 \,\in\, Y$, the Brauer class $\beta$ can be described as the class of the projective
bundle ${\mathbb P}_{p_0}$ over $U^{' s}_{Y, par}(r, \,L)$ whose fiber is $\mathbb{P}(E_{p_0})$ over
any $E_* \,\in \,U^{' s}_{Y, par}(r,\, L)$, in other words, $\beta$ is given by the Azumaya algebra with fibers ${\rm End}(E_{p_0})$. 

By \cite[Theorem 1.2]{B6}, we have the following result over $Y_{\mathbb C}$.

\begin{theorem}\label{brgpC} 
Assume that $g \,\ge \,2$, and if $g \,=\, 2\,=\, r$ then assume that $d$ is odd. Then
$$Br(U^{' s}_{Y_{\mathbb C}, par}(r,\, L_{\mathbb C})) \,=\, \mathbb{Z} / m \mathbb{Z} \, ,$$
where 
\begin{equation}\label{m}
m \,=\, {\rm g.c.d.} (r,\, d,\, \{k_1(x),\, k_2(x),\, \cdots,\, k_{l_x}(x)\}_{x\in I} )
\end{equation}
(see \eqref{rk}). Moreover, the Brauer group ${\rm Br}(U^{' s}_{Y_{\mathbb C}, par} (r,\, L_{\mathbb C}))$ is generated by the class
$\beta$ in Definition \ref{brauerclass}.
\end{theorem}

\begin{proposition} \label{exponent}
Let $L$ be a real point of the Picard variety ${\rm Pic}^d(Y)$. Then
$$\beta \,\in\, {\rm Br}(U^{' s}_{Y, par}(r,\, L))$$ has exponent $m$.
\end{proposition}

\begin{proof}
This follows from Proposition \ref{picMpar} and \cite[Lemma 3.10(v), Lemma 3.9]{H}. 
\end{proof}

\subsection{The Leray spectral sequence}

Let $$\psi\,:\, U^{'s}_{Y, par}(r, \,L)_{\mathbb C} \,=\, U^{'s}_{Y_{\mathbb C}, par}(r,\, L_{\mathbb C})
\,\longrightarrow \,U^{'s}_{Y, par}(r,\, L)$$ be the projection. Let
$$\psi_{p,2}\,:\, {\rm Br}(U^{' s}_{Y, par}(r,\, L)) \,\longrightarrow\, {\rm Br}(U^{' s}_{Y_{\mathbb C}, par} (r,\, L_{\mathbb C}))$$
be the homomorphism induced by $\psi$. Since $\psi_{p,2}(\beta) \,=\, \beta_{\mathbb C}$, the map $\psi_{p, 2}$ is surjective.

Let
$$f\,:\, U^{' s}_{Y, par}(r,\, L) \,\longrightarrow \,\mathbb{ R}$$
be the structure morphism. The Leray spectral sequence associated to it is defined by
$$
E^{p,q}_2 \,\,=\,\, H^p_{\rm et} (\mathbb{R},\, R^q f_* \mathbb{G}_m) \,\Rightarrow\, 
H^{p+q}_{\rm et} (U^{' s}_{Y, par}(r,\, L),\, \mathbb{G}_m)\, .
$$
It gives a short exact sequence in lower terms 
\begin{equation} \label{spectral}
H^1\, \longrightarrow\, E^{0, 1}_2 \, \longrightarrow\, E^{2, 0}_2 \,\longrightarrow\,
{\rm kernel}\{H^2\, \longrightarrow\, E^{0, 2}_2\}\, \longrightarrow\, E^{1,1}_2\, .
\end{equation}

By Corollary \ref{functions}, $H^0(U^{' s}_{Y_{\mathbb C}, par}(r, \,L)_\mathbb{C},\, \mathcal{O})\,=\, \mathbb{C}$. 
It follows that $H^0(U^{' s}_{Y_{\mathbb C}, par}(r, \,L)_\mathbb{C}, \,\mathcal{O}^*)\,= \,\mathbb{C}^*$.
The natural map $$\mathbb{G}_m \,\longrightarrow\, f_*f^* \mathbb{G}_m \,=\, f_* \mathbb{G}_m$$ is an isomorphism, and hence
$f_* \mathbb{G}_m \,=\, \mathbb{G}_m$.

We have $R^1 f_*\mathbb{G}_m \,=\, \ {\rm Pic}(U^{' s}_{Y_{\mathbb C}, par}(r,\, L)_{\mathbb C})$ and
$R^2 f_*\mathbb{G}_m \,=\, {\rm Br}(U^{' s}_{Y, par}(r,\, L)_{\mathbb C})$. Also, $$E^{1,1}_2 \,= \,H^1_{et} (\mathbb{R},\, \mathbb{Z})
\,=\,0$$ 
(see \cite[Proof of Theorem 3.3]{BHHS} for details). Therefore the spectral sequence \eqref{spectral} gives an exact sequence 
\begin{equation}\label{es}
{\rm Pic}(U^{' s}_{Y, par}(r,\, L)) \ \stackrel{\psi_{p,1}}{\longrightarrow} \ {\rm Pic}(U^{' s}_{Y_{\mathbb C}, par}(r, \,L_{\mathbb C}))
\ \longrightarrow\ {\rm Br}({\mathbb R})
\end{equation}
$$ \stackrel{f^*}{\longrightarrow} \ {\rm Br}(U^{' s}_{Y, par}(r,\, L))\ \stackrel{\psi_{p,2}}{\longrightarrow}\
{\rm Br}(U^{' s}_{Y_{par}}(r,\, L)_{\mathbb C}) \ \longrightarrow\ 0\, .$$

By Theorem \ref{brgpC}, ${\rm Br}(U^{' s}_{Y_{\mathbb C}, par}(r,\, L_{\mathbb C})) \,\cong\, {\mathbb Z}/m{\mathbb Z}$ and
it is generated by $\beta_{\mathbb C}$. Also, we have ${\rm Br}(\mathbb{R}) \,=\, \mathbb{Z}/ 2\mathbb{Z}$; the non-trivial
element $[\mathbb{H}]$ of it is the class of the quaternion algebra $\mathbb{H}$. Hence the exact sequence in \eqref{es} becomes 
\begin{equation}\label{seq2}
0 \,\, \longrightarrow \ {\rm Cokernel} \,\, \psi_{p,1} \,\, \longrightarrow\,\, \mathbb{Z}/ 2\mathbb{Z}
\,\, \stackrel{f^*}{\longrightarrow}\,\, {\rm Br}(U^{' s}_{Y, par}(r,\, L))\,\, \stackrel{\psi_{p,2}}{\longrightarrow}
\,\, {\mathbb Z}/m{\mathbb Z} \,\, \longrightarrow\,\, 0\, .
\end{equation} 

Note that $m$ defined in \eqref{m} can also be written as 
$$ m \,=\, {\rm g.c.d.} (r,\, \chi, \,\{c^x_2,\, \cdots, \,c^x_{l_x}\}_{x\in I} ) \, .$$

\begin{theorem}\label{brauergroup}
Assume that either $g \,\ge\, 3$ or $g\, =\, 2,\, r \,\ge\, 3$. Then 
$${\rm Br}(U^{' s}_{Y, par}(r,\, L)) \ = \ \mathbb{Z} \beta \oplus f^*({\rm Br}(\mathbb{R})) \ \cong \ 
\frac{\mathbb{Z}}{m {\mathbb Z}} \oplus \frac{\mathbb Z}{2 \mathbb{Z}}\, .$$
\end{theorem}

\begin{proof}
The proof is on similar lines as that of \cite[Theorem 1.3]{B6}, \cite[Theorem 3.3]{BHHS}.
We first make an observation. By \eqref{seq2}, the group ${\rm Cokernel} \ \psi_{p,1} \,\subset\, \mathbb{Z}/ 2\mathbb{Z}$. Hence if
it is nonzero, then ${\rm Cokernel} \ \psi_{p,1} \,=\, \mathbb{Z}/ 2\mathbb{Z}$. 

By Theorem \ref{picUYspar} and Remark \ref{picULpar}, we have $\ {\rm Cokernel} \ \psi_{p,1} \,=\, 0$. Since both $\beta$ and $\beta_{\mathbb C}$ have the same exponent $m$ (Proposition \ref{exponent}(1)), the sequence \eqref{seq2} splits so that 
$${\rm Br}(U^{' s}_{Y, par}(r,\, L)) \,= \,\mathbb{Z} \beta \oplus f^*({\rm Br}(\mathbb R))
\,= \,\mathbb{Z}/ m\mathbb{Z} \oplus \mathbb{Z}/ 2\mathbb{Z}\, .$$
This completes the proof.
\end{proof}

\section*{Acknowledgements} 

We are very grateful to the referee for pointing out a crucial error in an earlier version, and also for providing numerous suggestions
to improve the exposition.
This work was done during the tenure of the first author as INSA Senior Scientist in Indian Statistical Institute, Bangalore.
A part of this work was done during a visit of the first author to TIFR, Mumbai to attend Discussion Meeting on vector bundles in 
March-April 2022, she would like to thank T.I.F.R. for hospitality. The second author is partially supported by a
J. C. Bose Fellowship.


\end{document}